\documentclass[a4paper,11pt]{amsart}
\usepackage[mathscr]{euscript}
\usepackage[all]{xy}
\usepackage{enumitem}

\usepackage{amsmath,amssymb,amsfonts,amscd}

\theoremstyle{plain}
\newtheorem{theorem}{Theorem}[section]
\newtheorem{lemma}[theorem]{Lemma}

\newtheorem{corollary}[theorem]{Corollary}
\newtheorem{definition}[theorem]{Definition}

\newtheorem{example}[theorem]{Example}
\theoremstyle{definition}
\newtheorem{remark}[theorem]{Remark}

\newtheorem{remarks}[theorem]{Remarks}

\newcommand\<[1]{\langle\,#1\,\rangle}
\newcommand\norm[1]{\Vert #1 \Vert}

\newcommand{\mB}{\mathscr{B}}

\newcommand{\mA}{\mathscr{A}}

\newcommand{\la}{\langle}
\newcommand{\ra}{\rangle}
\newcommand{\Cf}[1]{\ensuremath \mbox{\large${\mathbf{1}}$}_{#1}}
\def\wasabi{wscai }
\newcommand{\bd}{{\ast\ast}}
\newcommand\restr[2]{{
  \left.\kern-\nulldelimiterspace 
  #1 
  \vphantom{|} 
  \right|_{_{#2}} 
  }}

\newcommand{\C}{\mathbb{C}}
\newcommand{\T}{\mathbb{T}}
\newcommand{\N}{\mathbb{N}}

\newcommand{\A}{\mathscr{A}}

\newcommand{\G}{\mbox{\tiny $G$}}
\newcommand{\dmg}{\ensuremath  d\mathrm{m}_{\G}}
\newcommand{\mg}{\ensuremath  \mathrm{m}_{\G}}

\newcommand{\mJ}{{\mathcal J}}

\usepackage{tikz-cd}
\usepackage{mathtools}

\begin{document}

\title[A note on a question of Garth Dales]{A note on a question of Garth Dales: Arens
regularity as a three space property}

\author[Filali, Galindo ]{Mahmoud Filali,  Jorge Galindo}

\address{Mahmoud Filali. Department of Mathematical Sciences,
    University of Oulu, Oulu Finland;\hfill\break \noindent E-mail:  mfilali@cc.oulu.fi}

\address{ Jorge Galindo. Instituto Universitario de Matem\'aticas y Aplicaciones (IMAC),
Universitat Jaume I, E-12071, Castell\'on, Spain;\hfill\break \noindent E-mail:  jgalindo@uji.es}




\begin{abstract}
Garth Dales asked whether a Banach algebra $\mA$ with an Arens regular closed ideal $\mJ$
and an Arens regular quotient $\mA/\mJ$ is necessarily Arens regular. We prove in this note  that, for a class of Banach algebras including the standard
algebras in harmonic analysis, Garth's conditions force the algebra to be  even reflexive.
We also give examples of Banach algebras with Garth's conditions, that are not Arens regular.

\end{abstract}
\dedicatory{To the memory of Harold Garth Dales}
\maketitle

\section{Introduction}
Before Garth Dales sadly  passed away, he asked   the authors \cite{garthlet}, and presumably other colleagues, whether Arens regularity was  a three-space property, that is, whether  a Banach algebra $\mA$ with an Arens regular closed ideal $\mJ$ yielding  an Arens regular quotient $\mA/\mJ$ is necessarily Arens regular.

In this note, we first consider  weakly sequentially complete Banach algebras $\mA$ which are    ideals in their  second dual $\mA^{**}$ and   have a contractive approximate identity.
We call such algebras \emph{wscai algebras}.
 When  such a Banach algebra is a direct summand of  its multiplier algebra, i.e., when $M(\mA)=\mA\oplus B$ where $B$ is a closed subalgebra of
$M(\mA)$,
we prove that Garth's question has a positive answer, forcing  $\mA$ to be  even reflexive.

 We then provide  examples  demonstrating  that, in general, Arens regularity fails the    three-space property. We see that such  examples  can be found in algebras that exhibit two of the following properties: (1) $\mA$ is an  ideal in $\mA^\bd$, (2) $\mA$ is weakly sequentially complete and (3) $\mA$ has an (approximate) identity.

Following \cite{arens51,arens51mo}, the first Arens product on $\mA^{**}$ is defined  through the following three stages. For $f\in \mA^*,$ $\varphi,\psi\in \mA$ and $m,n\in \mA^{**},$
define
\begin{equation}\label{260823a}
\langle f\cdot \varphi,\psi\rangle= \langle f, \varphi\psi \rangle,\quad
\langle n\cdot f,\varphi\rangle= \langle n, f\cdot \varphi \rangle\quad
\langle m n, f\rangle=\langle m, n\cdot f\rangle.
\end{equation}
 In an   analogous way,  the \emph{second Arens product} in $\mA^{**}$
 is defined  via the following relations for $f\in \mA^*,$ $\varphi,\psi\in \mA$ and $m,n\in \mA^{**},$
\begin{equation}\label{030923a}
\la \varphi \cdot f, \psi \ra= \la f, \psi \varphi \ra, \quad \la f\cdot m , \varphi \ra =\la m , \varphi \cdot f \ra , \quad \la m \diamond n , f\ra =\la n  ,f\cdot m \ra .
\end{equation}
With either of these products, the Banach space $\mA^\bd$ is made into a Banach algebra containing $\mA$ as a subalgebra.

It is clear from the way the products are defined that the translations
\[m\mapsto m n\quad\text{and}\quad m\mapsto n\diamond m: \mA^{**}\to \mA^{**}\]  are weak$^\ast$-weak$^\ast$-continuous for
 any $n\in \mA^\bd$,
and the translations
\begin{equation}\label{regular} m\mapsto nm\quad\text{and}\quad m\mapsto m\diamond n: \mA^{**}\to \mA^{**}\end{equation} are weak$^\ast$-weak$^\ast$-continuous for  any $n\in \mA$.
The algebra $\mA$ is {\it Arens regular} when one of the latter translations (and then both are) is weak$^\ast$-weak$^\ast$-continuous not only for the elements $n\in \mA$ but for every
$n\in \mA^{\bd}$. This is the same as having $mn=m\diamondsuit n$, for every $m, n \in \mA^\bd$.
When the continuity of each of the two translations in (\ref{regular})   holds for no $n\in \mA^\bd \setminus \mA$, we say that $\mA$ is \emph{strongly Arens irregular}.


It may be worth remarking that  an algebra $\mA$ may be  both Arens regular and strongly Arens irregular.
This will be the case  if and only if $\mA$ is reflexive.

The relevance of the space $WAP(\mA)$ of {\it weakly almost periodic functionals} on $\mA$ was first noticed, independently, by    Pym and Young   in \cite{pym} and \cite{young73}. This  is the space of all $f\in \mA^*$ such that the linear
map
\[\varphi \mapsto  \varphi \cdot f:\mA\to\mA^*\]
is weakly compact.  A consequence of  Grothendieck's criterion is that  $WAP(\mA)$ is actually the subspace of  $\mA^\ast$  which consists of those $f\in \mA^\ast$ such that   \[\langle m  n , f\rangle = \langle m \diamondsuit  n , f\rangle
\quad\text{for every} \quad m, n \in \mA^{**}.\]
So, $\mA$ is
Arens regular if and only if  $WAP(\mA) = \mA^{*}$.

By a bounded right approximate identity (\emph{brai} for short), we mean a bounded net $ \{e_\alpha\}_\alpha$ in $\mA$ such that for all $\varphi \in \mA$, $\varphi e_\alpha\to \varphi$ in $\mA$.
 A bounded left
approximate identity (blai) is defined similarly. A bounded approximate
identity (bai) is a net which is both a blai and a brai. When the terms of an approximate identity are all contained in the unit ball of $\mA$, we call it  \emph{contractive}.  The accumulation points of a brai $ \{e_\alpha\}_\alpha$ in $\mA^\bd$ are right identities in $\mA^\bd$  for the first Arens product, and the accumulation points of a  blai are left identities in $\mA^\bd$ for the second Arens product. So, the accumulation points of bais are right identities for the first Arens product and a left identities for the second, such identities are
known as mixed identities.

Suppose that  $ \mA $ is a weakly sequentially complete Banach algebra that is an ideal in  $ \mA^{**} $ and contains a contractive approximate identity  and let $ e $ be a mixed identity of $ \mA^{**} $, i.e., $\mA$ is a \wasabi algebra.
The second dual $\mA^{**}$ of $\mA $ has then  the direct sum decomposition \[\mA^{**}=e\mA^{**}\oplus WAP(\mA)^\perp,\] with
\begin{align*}WAP(\mA)^\perp&=\left\{m\in \mA^\bd \colon \<{m,f}=0, \mbox{ for any $f\in WAP(\mA)$}\right\}\\&
= (1-e)\mA^{**}.\end{align*} The subspace  $WAP(\mA)^\perp$ is actually  a closed ideal in $\mA^{**}$ and
\[\mA^{**}\,WAP(\mA)^\perp=WAP(\mA)^\perp\diamondsuit \mA^{**}=\{0\}.\] The Banach algebras
$e\mA^{**}$,   $WAP(\mA)^*$ and
 the multiplier algebra $M(\mA)$ of $\mA$ are isometrically isomorphic.  Nevertheless, it may be worthwhile to recall the educating remark
put by Baker, Lau and Pym in \cite[page 190]{BLP}, that $e\mA^{**}$ and $WAP(\mA)^*$ (and so $M(\mA)$)  need not be  isomorphic
when they carry their natural   weak$^*$ topologies induced, respectively  from $\mA^{*}$  and $WAP(\mA)$.
 For the details, see \cite{BLP}, \cite{dal00} and \cite{EFG2}.

As noted  in \cite{EFG2},  if $\mJ$ is a closed ideal in $\mA,$ then $\mJ^{**}$ admits an analogous decomposition. Denoting by $i^\ast \colon \mA^\ast\to\mJ^\ast$  the restriction map, adjoint to the inclusion map $i\colon \mJ \to \mA$, one has

 \begin{equation}\label{desc}\mJ^{**}=e\mJ^{**}  \oplus i^*(WAP(\mA))^\perp=e\mJ^{**} \oplus (\mA\cdot\mJ^*)^\perp,\end{equation}
where now \begin{align*}i^*(WAP(\mA))^\perp&=\left\{m\in \mJ^\bd \colon \<{m,i^*(f)}=0, \mbox{ for every $f\in WAP(\mA)$}\right\}\\&=
\left\{m\in \overline{\mJ} \colon \<{m,f}=0, \mbox{ for every $f\in WAP(\mA)$}\right\},\end{align*}
and $\overline \mJ$ denotes the weak$^*$-closure of $\mJ$ in $\mA^{**}$.
 The right ideal  $e\mJ^{**}$ may be identified with $\overline{\mJ}^w$, the weak$^*$-closure of
$\mJ$ in $WAP(\mA)^*$. The properties of this decomposition, discussed in \cite{EFG2}, are summarized below\label{page 3intro}
\begin{enumerate}
\item $\overline{\mJ}^w $ is a closed ideal of $M(\mA),$
\item $i^*(WAP(\mA))^\perp$ is a closed ideal of $\mJ^{**},$
\item  $\overline \mJ^w $ is isometrically  isomorphic to  both     $  e\mJ^{**}$  and $\mJ^{**}\diamondsuit e,$
\item $i^*(WAP(\mA))^\perp=(1-e)\mJ^{**}=\mJ^{**}\diamondsuit (1-e),$
\item $\mJ^{**}\,i^*(WAP(\mA))^\perp=i^*(WAP(\mA))^\perp\diamondsuit \mJ^{**}=\{0\}.$
\end{enumerate}

%

\subsection{Some special ideals of $L^1(G)$}
During the following discussion $G$ will denote a compact Abelian group  with Haar measure $\mg$  and $\widehat{G}$ the  group of all continuous homomorphisms of $G$ into the multiplicative group $\T$ of all complex numbers of modulus 1. If $M(G)$ denotes the space of  bounded Borel  measures on $G$  and $  \mu \in M(G)$,  the \emph{Fourier-Stieltjes transform} of $\mu$ is the bounded function $\widehat{  \mu}\colon \widehat{G}\to \C$ defined by \[\widehat{\mu}(\chi)=\int_{_G} \chi(t^{-1})d\mu(t) \mbox{ for each $\chi \in \widehat{G}$.}
\]
If $ X$ is a linear subspace of $M(G) $ and  $ E \subset \widehat G $, we denote by $X_E  $   the  subspace of $ X $ consisting of all elements in $X$ with  Fourier-Stieltjes  transform supported in $E$, that is
\[X_E   = \lbrace  \mu \in X:  \widehat{\mu}(\chi) =0 ~ \text {for}~ \chi \in \widehat{G}\setminus E \rbrace.\]
We will always regard  $L^1(G)$ as an ideal in $M(G)$ through the embedding $f\mapsto f\cdot \dmg$.

In Harmonic Analysis, there has  been  a great interest in subsets of $\widehat G$ making $X_E=Y_E$ for some subspace  $Y$ of $X.$
For instance {\it Riesz sets} are those for which  $ M_E(G) = L^1_E(G)$,
and \emph{ Rosenthal sets } are those for which $L^\infty_E(G)=C_E(G)$. We will say that  $E\subseteq \widehat{G}$ is a  \emph{co-Rosenthal set} when its complement $\widehat G\setminus E$ is a Rosenthal set,
that is, when,    $f\in L^\infty(G)$ and  $\widehat{f}(E)=\{0\}$ imply that $f$ is continuous.

We will be  assuming      that $L^\infty(G)$ is defined so that $L^1(G)^\ast=L^\infty(G) $, through the duality
\[ \<{f,\varphi}=\int \varphi(x)\overline{f(x)}\dmg(x), \quad f\in L^\infty(G),\;\varphi\in L^1(G) .\]
Since   $\widehat{f}(\chi)=\<{f,\chi}$, for every $f\in L^\infty(G)$ and $\chi \in \widehat{G}$,
and since  $E\subseteq L^1_E(G)$, it is clear that $L^1_E(G)^\perp\subseteq L^\infty_{\widehat G\setminus E}(G)$. On the other hand,  we may $L^1$-approximate
the functions $\varphi$ in  $L^1_{E}(G)$ by trigonometric polynomials $\varphi_n$ spanned by characters from $E$. Since $\<{f, \varphi_n}=0\quad\text{ for every}\quad f\in L^\infty_{\widehat G\setminus E}(G),$
we see that $\<{f, \varphi}=0\quad\text{ for every}\quad f\in L^\infty_{\widehat G\setminus E}(G)$.
As a consequence, \[L^1_E(G)^\perp=\{f\in L^\infty(G):\la f ,\varphi\ra=0\;\text{for all}\;\varphi\in L^1_{E}(G)\}= L_{\widehat{G}\setminus E}^\infty(G),\]
and, accordingly,  $E$ is a   co-Rosenthal set if and only if
$L^1_E(G)^\perp \subseteq C(G).$
Now, it happens that a functional on $L^1(G)$ is  weakly almost periodic if and only if it is continuous (one direction is a theorem of    \"Ulger \cite{ulger} valid for any locally compact group, the other only works for compact groups and  can be deduced from Grothendieck's double limit criterion), so 
\[L^1_E(G)^\perp\subseteq WAP(L^1(G)),\]
whenever  $E$ is a co-Rosenthal set.
%
%

%
\subsection{Riesz and co-Rosenthal ideals}
Closed ideals of  $L^1(G)$ are always of the form $L^1_E(G)$ for some subset $E$
in $\widehat G$ (see \cite[Theorem 38.7]{HR70}). We may therefore describe the ideals of  $L^1(G)$ through the properties of the  corresponding set in $\widehat G$
 and  may rightfully speak of  Riesz or  co-Rosenthal \emph{ideals} of $L^1(G)$.
 Under this point of view,  bearing in mind the discussions earlier in this section, the following definitions
are
 genuine extensions of the concepts of  Riesz and co-Rosenthal set.

\begin{definition}\label{def1}
Let $\mA$ be a Banach algebra  and let $\mJ$ be a closed ideal of $\mA$.
  \begin{itemize}
\item[--] We say that $\mJ$ is \emph{co-Rosenthal} if $\mJ^\perp \subseteq WAP(\mA).$
\item[--] Assume in addition that  $\mA$ has  a  mixed identity $e\in  \mA^{**} $, we say that $\mJ$ is a  \emph{ Riesz ideal} if  $e\mJ^{**}  = \mJ$.
\end{itemize}
\end{definition}

 \begin{remark}
   The definition of Riesz ideal does not depend on the choice of the mixed identity of $\mA^\bd$. Indeed, once one knows that $e\mJ^{**}=\mJ$, for some right identity $e\in \mJ^\bd$, then $f\mJ^\bd=fe\mJ^\bd=f\mJ=\mJ$  for any other mixed identity $f\in \mJ^\bd$.

      When $\mA$ has a contractive approximate identity, $\mJ$ being Riesz turns to be equivalent to being weak$^\ast$-closed in the multiplier algebra $M(\mA)$, see the identifications of $M(\mA)$ with $e\mA^\bd$ and $WAP(\mA)^\ast$ described  in the Introduction. So, if  $G$ is a   compact group and $\mA=L^1(G)$, $\mJ=L_E^1(G)$, $E\subseteq \widehat{G}$, is Riesz when $\mJ$ is weak$^\ast$-closed in $M(G)$, and so $L_E^1(G)=M_E(G)$. In the same vein, when $G$ is a  discrete amenable group and $E\subseteq G$, an ideal $ A_E(G)=\{u\in A(G) \colon u(G\setminus E)=\{0\}\}$  is Riesz if it is weak$^\ast$-closed in $B(G)$, and so $A_E(G)=B_E(G)$.
 \end{remark}

\begin{lemma} \label{co-R} For a closed ideal $\mJ$ in $\mA$, the following statements are equivalent.
\begin{enumerate}
\item $\mJ$ is co-Rosenthal.
\item $i^{**}\bigl(i^*(WAP(\mA))^\perp\bigr)=WAP(\mA)^\perp$.
\item The quotient Banach algebra $\mA/\mJ$ is Arens regular.
\end{enumerate}
\end{lemma}

\begin{proof} Let $m\in i^*(WAP(\mA))^\perp$.
Then \[\la i^{**}(m),f\rangle=\<{ m,i^*(f)}=0,\quad\text{for every}\quad f\in WAP(\mA),\] i.e., $i^{**}(m)\in WAP(\mA)^\perp.$ Hence, the inclusion $i^\bd\bigl(i^*(WAP(\mA))^\perp\bigr)\subseteq WAP(\mA)^\perp $  always holds.

Assume now that  $\mJ$ is co-Rosenthal and let  $m\in WAP(\mA)^\perp$. Then $m\in \mJ^{\perp \perp}=i^{**}(\mJ^{**})$, so that $m=i^{**}(n)$
for some $n\in \mJ^{**}$ and  \[\langle n, i^*(f)\rangle=\langle m, f\rangle=0\quad\text{for every}\quad f\in WAP(\mA).\] We see that $n\in i^*(WAP(\mA))^\perp$
and so
$m\in i^{**}\bigl(i^*(WAP(\mA)^\perp)\bigr)$. Therefore, (i) implies (ii).

To check that (ii) implies (i), suppose otherwise that $\mJ^\perp$ is not contained in $WAP(\mA)$. Then there is some element $m\in WAP(\mA)^\perp$ which is not in $\mJ^{\perp\perp}=i^{**}(\mJ^{**})$.
In particular, $m$ is not in $i^{**}( i^*(WAP(\mA))^\perp),$ and so (ii) fails to hold.
\medskip

Let now  $\pi:\mA\to \mA/\mJ$ be the quotient map, and recall that $\mJ^\perp$
and $(\mA/\mJ)^*$ are identified via  $f=\widetilde{f}\circ\pi,$
where $f\in \mJ^\perp$ and $\widetilde{f}\in (\mA/\mJ)^*$.
Then for any two nets
$(\varphi_\alpha)$ and $ (\psi_\beta) $ in $\mA$ we find
\[\langle f, \varphi_\alpha \psi_\beta\rangle=\langle\widetilde f, \pi(\varphi_\alpha)\pi(\psi_\beta)\ra.\]
Since $\mA/\mJ$ is Arens regular if and only if $WAP(\mA/\mJ)=\mJ^\perp$, this quickly shows that Statements (i) and (iii) are equivalent.
\end{proof}

\begin{corollary}
  Let $G$ be a locally compact Abelian group and $E\subseteq  \widehat{G}$.  The following statements are equivalent.
\begin{enumerate}
\item $\widehat{G}\setminus E$  is a Rosenthal set, i.e., $E$ is a co-Rosenthal set.
\item The quotient Banach algebra $L^1(G)/L_E^1(G)$ is Arens regular.
\end{enumerate}
\end{corollary}

We refer to \cite[Lemma 4.5]{EFG2} for the proof of the following lemma.
\begin{lemma}  \label{lem5}
Let $ \mA $ be \wasabi Banach algebra. Let $ m \in \mA^{**}$ be such that $\mA^* \cdot m  \subseteq WAP(\mA)$ or  $m \cdot \mA^* \subseteq WAP(\mA)$. Then $m =\varphi_0 +r$ for some $\varphi_0\in \mA$ and $r\in WAP(\mA)^\perp$.
\end{lemma}
%
%

We record here the following elementary facts  for later use.

\begin{lemma}\label{straight} For $f \in\mA^*$, $e\in\mA^{**}$ and $m\in \mJ^{**}$, we have
\begin{enumerate}
\item $m\cdot i^*(f)=i^*(m\cdot f)$.
\item $(em)\cdot f=e\cdot(m\cdot f)$.
\end{enumerate}
\end{lemma}

\begin{lemma}\label{1.5} Let $ \mA $ be a \wasabi Banach algebra and $\mJ$ a closed ideal in $\mA$.
If $\mJ$ is co-Rosenthal and Arens regular, then it is Riesz.
\end{lemma}

\begin{proof}  Let   $(e_\alpha)_\alpha $ be a contractive approximate identity  in $\mA$ and $e$ be   a mixed identity in $ \mA^{**} $ associated with
$ (e_\alpha)_\alpha $.

Let   $m\in \mJ^{**}$ and fix $ f\in \mA^\ast$. Given any $r\in i^*(WAP(\mA))^\perp$, using Lemma \ref{straight}, one has
\begin{align*}
\langle r, i^*((em)\cdot  f)\rangle&=\langle i^{**}(r), (em)\cdot  f\rangle =\langle i^{**}(r), e\cdot(m\cdot  f)\rangle \\&=
\langle i^{**}(r)e, m\cdot  f\rangle
= \langle i^{**}(r), m\cdot  f\rangle=\langle r, i^*(m\cdot  f)\rangle\\&=
\langle r,m\cdot i^*( f)\rangle=\langle rm, i^*( f)\rangle.
\end{align*}
Taking into account  that $\mJ$ is assumed to be Arens regular, we find  \[\langle r m, i^*( f)\rangle=\langle r \diamondsuit m, i^*( f)\rangle=0.\] Therefore, $\<{r,i^\ast((em)\cdot  f)}=0$. Since this equality holds  for every $r\in i^*(WAP(\mA))^\perp$, we conclude that $ i^\ast(em\cdot  f)\in \overline{i^*(WAP(\mA))}$.

Let then $(g_n)_n$ be a sequence in  $ WAP(\mA)$   such that
$\lim_n i^*(g_n)=i^*((em)\cdot f)$.
Since the restriction mapping $i^\ast$ is a quotient map, one can find (see, e.g., the proof of  \cite[Theorem4.2(b)]{conway}) a sequence $ (h_n) $ in $ \mA^\ast $ and $ h_0 \in  \mA^\ast $ such that, for each $ n\in \mathbb{N} , i^*(g_n) = i^*(h_n) $, and $\lim_n h_n = h_0$.
As  $i^\ast(g_n) = i^*(h_n) $, one finds that  $g_n -h_n \in J^\perp\subseteq  WAP(\mA) $ and, hence, that    $h_n \in WAP(\mA)$ as well.  As a consequence, $h_0 \in WAP(\mA)$. The equality $i^*(em\cdot  f)=i^\ast(h_0)$   implies  that $em \cdot f\in h_0 + J^\perp \subseteq  WAP(\mA)$.

Since $em \cdot f\in WAP(\mA)$ for every $f\in \mA^\ast$, it follows from  Lemma \ref{lem5} that  $em=\varphi+r$ for some $\varphi \in\mA$ and $r\in WAP(\mA)^\perp$. Taking into account that $e\mA^{**}\cap WAP(\mA)^\perp=\{0\}$, we must have $em=\varphi$, which means that  $em\in \mA\cap \mJ^\bd$. As  elements in $\mA\cap \mJ^\bd$ can be obtained as weak$^\ast$-limits of nets in $\mJ$ and $\mJ$ is weakly closed in $\mA$, we conclude that $em\in \mJ$.

 The ideal, $\mJ$ is thus a Riesz ideal.
\end{proof}
To state our main theorem we require the following technical condition.

%
%

\begin{definition}
  We say that a   Banach algebra is a \emph{standard } \wasabi algebra, if it is in the \wasabi class  and  $e\mA^\bd$ contains a closed vector subspace $\mB$ such that  $e\mA^\bd=\mA\oplus_1 \mB$.
\end{definition}

As already remarked,  $e\mA^\bd$ is isometrically isomorphic to the multiplier algebra $M(\mA)$. The  Lebesgue decomposition theorem and its extension to the Fourier-Stieltjes algebra, see \cite{miao99} and \cite{kanilauschl03}, show that group algebras on compact groups and  Fourier algebras  on discrete amenable groups  are both standard \wasabi algebras.

\begin{theorem}\label{main}
Let $\mA$ have a bai and be an ideal in its second dual and let $\mJ$ be a closed ideal in $\mA$. Assume  that $e\mA^{**}=\mA\oplus_1 B$. If $\mJ$ is Riesz and co-Rosenthal, then $\mA$ is regular.
\end{theorem}

\begin{proof}
We prove first that the quotient
algebra $\mA/\mJ$ is weakly sequentially complete. Note to begin with that, $\mJ$ being both co-Rosenthal and Riesz,  Lemma \ref{co-R} yields
\begin{align*}(\mA/\mJ)^{**}=\mA^{**}/i^{**}(\mJ^{**})&=\frac{e\mA^\bd\oplus WAP(\mA)^\perp}{{i^{**}(e\mJ^{**})\oplus i^{**}(i^*(WAP(\mA))^\perp})}\\&=\frac{e\mA^\bd}{i^{**}(e\mJ^{**})}=\frac{e\mA^\bd}{\mJ}\\&=\frac{\mA\oplus_1 B}{\mJ}=\frac{\mA}{\mJ}\oplus_1 B.\end{align*}

 Define
$S\colon  (\mA/\mJ)^{**}\to  (\mA/\mJ)^{**}$ by \[S\bigl((u,v)\bigr)=(u,-v)\quad\text{with}\quad u\in \mA/\mJ\;\text{ and}\;v\in B.\] Then $S$ is an isometry and, if $I$ denotes the identity map on $( \mA/ \mJ)^\bd$, then $\ker(S-I)=\mA/\mJ$. By \cite[Corollary II.2]{godelust90}, $\mA/\mJ$ is weakly sequentially complete.

Since $\mA/\mJ$ has a bai, namely $(e_\alpha+\mJ)$, where $(e_\alpha)$ is a bai in $\mA,$ and $\mA/\mJ$ is clearly a closed ideal
in $(\mA/\mJ)^{**}$.
  By \cite[Theorem 2.1(iii)]{BLP}, this implies that $\mA/\mJ$ is strongly Arens irregular.
Since $\mA/\mJ$ is Arens regular by Lemma \ref{co-R}, $\mA/\mJ$ must be reflexive.
Therefore, \[\mA/\mJ=(\mA/\mJ)^{**}=\mA/\mJ\oplus_1 B,\] and so $B$ must be trivial. Thus $e\mA^\bd=\mA$ so that $\mA$ is a  Riesz ideal in itself. By \cite[Theorem 5.2]{EFG2} (see, as well,  \cite{Ulger11} and \cite{EFG1} for the particular case $\mA=L^1(G)$)
 $\mA$ is  Arens regular, as required.
\end{proof}

We can now answer Garth's question for \wasabi  algebras.

\begin{corollary} Let $\mA$ be a standard \wasabi algebra and $\mJ$ be a closed ideal in $\mA$.
If both $\mJ$ and $\mA/\mJ$ are Arens regular, then $\mA$ is reflexive.
\end{corollary}

\begin{proof} If $\mA/\mJ$ is Arens regular, then by Lemma \ref{co-R}, $\mJ$ is co-Rosenthal.
If in addition $\mJ$ is Arens regular, then by Lemma \ref{1.5}, $\mJ$ is Riesz.
So Theorem \ref{main} shows that $\mA$ is Arens regular. But since $\mA$ is a \wasabi algebra,
 \cite[Theorem 2.1(iii)]{BLP}  implies that $\mA$ is sAir. Thus, $\mA$ is reflexive.
\end{proof}

Next Corollary follows directly from the preceding one.
\begin{corollary} \label{Garth} A non-reflexive  standard \wasabi Banach  algebra $\mA$      cannot contain  a closed ideal $\mJ$ such that both $\mJ$ and the quotient $\mA/\mJ$ are  Arens regular.
\end{corollary}

\begin{remarks} (a) \"Ulger asked in \cite{Ulger11} whether Arens regular closed ideals in the group algebra of a compact abelian group are necessarily Riesz ideals.  The question was re-stated in \cite{EFG2} for any \wasabi  algebra.
If the algebra is not weakly sequentially complete, then examples for which this is not true are easily found, for example
$c_0$ is a closed ideal in $\ell_\infty,$ has a bai, is Arens regualar. But $c_0$ is not a Riesz ideal in $\ell_\infty$

(b)
It is known \cite{drespigno75} that in the dual $\widehat{G}$ of a compact Abelian group $G$, the union of a Riesz set and a Rosenthal set is a Riesz set. Since  $\widehat{G}$ itself is not Riesz unless $G$ is finite, a  Riesz subset of an infinite group can never be co-Rosenthal.

  By this reason, Lemma \ref{1.5} can be restated in this context.
	\begin{corollary}\label{Garth2}
A co-Rosenthal ideal of $L^1(G)$, $G$ compact infinite, is never Arens regular.
\end{corollary}
		One can but wonder whether this result holds true for more general  \wasabi algebras.

  Corollary \ref{Garth} can  be recast in  the light of  Corollary  \ref{Garth2}.
      \begin{corollary}
      Let $G$ be an infinite compact Abelian group.
      If $E\subseteq \widehat{G}$ is such that  $L^1(G)/L_E^1(G)$ is Arens regular, then $L_E^1(G)$ is not Arens regular.
          \end{corollary}
\end{remarks}

\section{Examples} In the previous section, we proved that
if a Banach algebra $\mA$ satisfies the following properties
\begin{enumerate}
\item $\mA$ is a closed two-sided ideal in $A^{**},$
\item $\mA$ contains a contractive approximate identity,
\item $\mA$ is weakly sequentially complete,
\item  for a mixed identity $e\in \mA^{**}$, the space $e\mA^{**}$ is a direct sum $\mA\oplus \mB$ for some closed vector subspace $\mB$,
\item $\mA$ contains a closed ideal $\mJ$ such that both $\mJ$ and $\mA/\mJ$ are Arens regular, i.e., $\mA$ satisfies
 Garth's conditions,
\end{enumerate}
then $\mA$ is reflexive, that is, $\mA$ agrees  strongly with Garth's conclusion.

In this section, we give   examples of Banach  algebras that satisfy Garth's conditions (v) but are not Arens regular. Of course, each of the algebras lacks some of the properties (i)--(iv), we have tried to keep as many of them as possible.

Our examples are built on certain \emph{semidirect} products constructed through the natural right and left dual actions of  a Banach algebra $\mB$ on its dual $ \mB^*$. For $x\in \mB^*$ and $\varphi,\:\psi\in \mB$, these actions  are given by
\begin{align*}
\tag{Right action}  \langle x\varphi, \psi\rangle&=\langle x, \varphi \psi\rangle\quad \quad  \mbox{ and }\\
\tag{Left action}\langle \varphi x, \psi\rangle&=\langle x, \psi\varphi\rangle.
 \end{align*}
Following the general procedure described by Arens in \cite{arens51} for forming adjoints of bilinear maps,
the right and left module actions of $\mB$ on $\mB^{*}$ extend to right and  left module actions of $\mB^{**}$ on $\mB^{***}.$
The extensions are  reached after three steps, akin to those leading   the Arens products.

 For any  $p\in \mB^{***}$ and
$m\in \mB^\bd$,  $pm$ is defined after defining $mf\in \mB^\bd$, $f\in \mB^\bd$  and $fx\in\mB^\ast$, $x\in \mB^\ast$:
\begin{align}\notag  \<{pm,f}&=\<{p,mf}, 
\\
\<{mf,x}&=\<{m,fx}, \;  \label{mf}
\\
  \<{ fx,\varphi}&=\<{f, x\varphi}, \quad \varphi \in \mB.\notag
\end{align}
The action and $ mp\in\mB^{***}$ is defined in the same vein \begin{align}\notag \<{mp,f}&=\<{m,pf},
\\
\<{pf,x}&=\<{p,fx}, \;  \label{pf}
\\
  \<{ fx,\varphi}&=\<{f, \varphi x}, \quad \varphi \in \mB.\notag\end{align}

We now define the Banach algebras that constitute our examples.
\begin{definition}
Let $\mB$ be a Banach algebra with dual $\mB^\ast$.  We define   three different Banach algebras with underlying Banach space $ \mB^*\oplus_1 \mB$.  The multiplications are defined as follows, for  $x,y\in \mB^*$, and $\varphi,\psi\in \mB$.
      \begin{description}
    \item[(i) The algebra  $\mathbf{\mB_{\mB^\ast}}$] $(x, \varphi)(y, \psi) = (x\psi , \varphi \psi)$,  if $\mB^\ast$ is regarded as a right $\mB$-module.
\item[(ii) The algebra $\mathbf{_{\mB^\ast}\mB}$]  $(x,  \varphi )(y, \psi) = ( \varphi y,  \varphi \psi)$,  if $\mB^\ast$ is regarded as a left $\mB$-module.
    \item[(iii)  The algebra    $\mathbf{_{\mB^\ast}\mB_{\mB^{\ast}}}       $]  $(x,  \varphi )(y, \psi) = (x\psi+  \varphi  y , \varphi\psi)$,  if $\mB^\ast $ is regarded  as a  $\mB$-bimodule.
\end{description}
\end{definition}

For $ p,q\in \mB^{\ast\ast\ast}$ and $ m,n\in \mB^\bd$, the first Arens product in these algebras  is then given by:
\begin{align*}
  \mathbf{(\mB_{\mB^\ast})^\bd}:  (p, m)(q,n) &= (pn, mn)\\
\mathbf{(_{\mB^\ast}\mB)^\bd}:  (p, m)(q,n)& = (mq, mn)\\
\mathbf{(_{\mB^\ast}\mB_{\mB^{\ast}})^\bd}: (p, m)(q,n) &= (pn+mq, mn),
\end{align*}
	 %
where $mn$ is the first Arens product of $m$ and $n$ in $\mB^\bd$ and $pn$ and $mq$ are given in \eqref{mf} and \eqref{pf}, respectively.

\begin{example}\label{leftideal} Let $\mB=c_0$  with pointwise product. The Banach algebra $\mA=_{\mB^\ast}\!\!\mB_{\mB^{\ast}}$ has the following properties:
\begin{enumerate}
      \item $  \mA$ is not weakly sequentially complete.
			\item $\mA$ is not standard.
  \item $\mA$ has a contractive approximate identity.
  \item $\mA$ is a two-sided  ideal in $\mA^\bd$.
	\item $\mA$ satisfies Dales's conditions.
	\item $\mA$ is not Arens regular.
\end{enumerate}
\end{example}

\begin{proof} The algebra  $c_0$ has a contractive approximate identity $(e_n)$, where $e_n=\Cf{\{1,...,n\}}$ for every $n\in\N$. It is a closed two-sided ideal in its second conjugate $\ell_\infty$. The  module action of $c_0$
on $\ell_1$ turns out to be pointwise multiplication, and so $\mA$ is a commutative Banach algebra.

The sequence $(0,e_n)$
is then a  contractive approximate identity  in $\mA$,  since \begin{align*}\lim_n\|xe_n-x\|_1=\lim_n \norm{\varphi e_n  -\varphi}_\infty=0 \quad\text{for every}\quad x\in\ell_1, \varphi\in c_0.
\end{align*}  To see that $\mA$ is a left ideal in $\mA^\bd,$  recall that \[(p,m)(x,\varphi)=(p\varphi+mx, m\varphi)\quad \text{for}\quad (x,\varphi)\in \mA,\; (p,m)\in \mA^\bd,\] and note that
$mx\in \ell_1$ and $m\varphi \in c_0$ since $c_0$ and $\ell_1$ are ideals in $\ell_\infty$ (under pointwise multiplication).
So we only need
to check that $p\varphi\in \ell_1$ for every $p\in \ell_1^\bd$ and $\varphi\in c_0.$
We decompose $p=x_0+p_0$ with $x_0\in \ell_1$ and $p_0\in c_0^\perp$. To see that $p\varphi \in \ell_1$
we consider a net $(f_\alpha)$ in $\ell_\infty$ that converges to 0 in the $\sigma(\ell_\infty, \ell_1)$-topology. Then, using that
$\varphi f_\alpha  \in c_0$ and that $x_0 \varphi\in \ell_1$,
\begin{align*}
\lim_\alpha \<{ p \varphi, f_\alpha}&=\lim_\alpha \<{p,\varphi f_\alpha}\\
&=\lim_\alpha \<{ x_0 \varphi, f_\alpha}+\lim_\alpha \<{p_0,\varphi f_\alpha}=0.
\end{align*}
It follows that $p\varphi \in \ell_1$,
as wanted.
Since $\mA$ is commutative, this shows that $\mA$ is also a right ideal in $\mA^{**}.$

To see that $\mA$ satisfies Dales' conditions, let $\mJ=\ell_1\times \{0\}.$
Then $\mJ$ is a closed two-sided ideal in $\mA$ since $(x,0)(y,\varphi)=(x\varphi,0)\quad\text{and}\quad (y,\varphi)(x,0)=(\varphi x,0)$,
and is clearly Arens regular since the operation in $\mJ$ is trivial.

The quotient $\mA/\mJ$ is isometrically isomorphic to  $c_0$. It  is, therefore, Arens regular.

  As is always the case with the dual action of a non-reflexive algebra with a brai \cite[Theorem 2.1]{EF}, the action of $c_0$ on $\ell_1$ is not regular, and therefore, neither is  $  \mA$.

	Since $c_0$ is not weakly sequentially complete, neither is the algebra $\mA$.

Let now $e=(0,\mathbf{1})$,   $\mathbf{1}$ standing for the constant 1-sequence in $\ell_\infty$,     be the  accumulation point of the approximate identity $(0,e_n)$ in $\mA^{\bd}$. Suppose  that $\mA$ is complemented in $e\mA^{\bd}$ and let $P\colon e\mA^\bd\to \mA$ be the corresponding projection. If $\pi_2\colon \mA=\mB^\ast \oplus\mB \to \mB$ is the second projection and we consider the embedding $i\colon \mB^\bd \to e\mA^\bd$ given by  $i(x)=(0,x)=(0,\mathbf{1})(0,x)$, then the composition of the following functions
\begin{align*}
  \mB^\bd \overset{i}{\to} e \mA^{\bd}\overset{P}{\to} \mA\overset{\pi_2}{\to} \mB
\end{align*}
is a bounded projection $\ell_\infty\to c_0$, whose existence  contradicts  the Phillips-Sobczyk theorem, see e.g. \cite[Theorem 2.5.5]{KA}.%
		\end{proof}
	\medskip
\begin{remark} If we consider the algebras $\mB_{\mB^\ast}$ or $_{\mB^{\ast}}\mB$, then the Banach algebra $\mA$ is not commutative, and  satisfies all the properties in Example \ref{leftideal} but the third one. It has just a contractive right (resp., left) approximate identity.

The element $(0,1)$ is just a right identity in $\mA^{**}$, and $(0,1)\mA^{**}=\{0\}\oplus_1\ell_\infty$ cannot be decomposed as $\mA\oplus_1 B$ since $c_0$ is not complemented in $\ell_\infty.$ So in this case, $\mA$ is not a standard algebra either. The rest of the statements follow in the same manner.
\end{remark}\medskip

To make Banach algebras modelled on spaces $\mB^\ast \oplus_1\mB$ weakly sequentially complete, one  needs nonreflexive Banach algebras that are  weakly sequentially complete and have
weakly sequentially complete duals. This is not easily achieved. Bourgain  and Delbaen managed to construct such a Banach space in  \cite{bourgdelb81}. One can define a Banach algebra structure on this    Banach space in  a standard, albeit quite trivial, way. The algebras $\mB_{\mB^\ast}$ and $_{\mB^\ast}\mB$ are then Banach algebras with   one-sided identity  which are 	one-sided ideals in their biduals.  These algebras satisfy Garth's conditions, but they are not Arens
	regular.

	Let   $\mB$ be the  weakly sequentially complete non-reflexive Banach space defined in \cite[Section 4]{bourgdelb81} whose  dual $\mB^*$ is   weakly sequentially complete as well.
		Fix $x_0\in \mB^*$ and $e\in \mB$ such that $\langle x_0,e\rangle=1.$
	We introduce an operation on $\mB$   by
	\begin{equation}\label{z1}\varphi  \psi=\langle x_0, \psi\rangle \varphi\end{equation}
turning it into a Banach algebra.

Then the following properties are easily verified.
\begin{enumerate}
\item  $e$ is a right identity for $\mB$.
\item $\mB$ cannot possess a blai $(e_\alpha)$ since $e_\alpha \varphi=\langle x_0,\varphi\rangle e_\alpha.$
\item  $\mB$ is a right ideal  but not a left ideal in its second dual.
\item  
 $\mB$  is Arens regular.
\end{enumerate}

The natural right module action of $\mB$ on $\mB^*$ is now given by
\begin{equation*}\label{natural}x\varphi= \langle x,\varphi\rangle x_0
\quad\text{for}\quad x\in \mB^*,\;\varphi\in\mB.
\end{equation*}

Let now $\mA= \mB_{\mB^\ast}$. The operation in $\mA$ simplifies to
\begin{equation*} \label{see}(x,\varphi)(y,\psi)  = (x\psi, \varphi\psi) =
( \langle x,\psi\rangle x_0,     \langle x_0,\psi\rangle \varphi),\end{equation*}
and so the first Arens product in $\mA^{**}$ is given by
\begin{align*} (p,m)(q,n)  & = (pn, mn) =
( \langle p,n\rangle x_0,     \langle n,x_0\rangle m).\end{align*}
From this identity, we see that $\mA$ is a right but not a left ideal in its second dual.
Note also that there is no net $(e_\alpha)$ in $\mB$ which acts as a brai for the right  $\mB$-module $\mB^*$,
since $\|xe_\alpha-x\|=\|\langle x,e_\alpha\rangle x_0-x\|$  converges to zero only if $x$ is a multiple of $x_0.$
This implies that $\mA$ does not possess a brai.
The algebra $\mA$ does not posses a blai either (use the observation (ii) above or argue directly).
And as in Example \ref{leftideal}, one  sees that $\mA$ satisfies Dales conditions.
Finally, since $\mB$ is not reflexive, it follows from \cite{EF}, that the natural right module action is not Arens regular, and so
the algebra $\mA$  is not Arens regular.
We summarize these observations in our second example.

\begin{example}\label{final} Let $\mA=\mB_{\mB^\ast}$ where $\mB$ is the space defined in \cite[Section 4]{bourgdelb81} with the operation (\ref{z1})
and the natural right module action of $\mB$ on $\mB^*.$ Then
\begin{enumerate}
\item $\mA$ is weakly sequentially complete.
\item $\mA$ is a    right but not a left ideal in its second dual.
\item  $\mA$ has neither a blai nor a blai.
\item $\mA$ satisfies Dales conditions.
\item $\mA$ is not Arens regular.
\end{enumerate}
\end{example}

\begin{remarks}
1) If we want the algebra $\mA$ in Example \ref{final} to be a left ideal in its second dual, then we take $\mB$ with the operation
\begin{equation}\label{z2}\varphi  \psi=\langle x_0, \varphi\rangle \psi,\end{equation}
 consider the left module action of $\mB$ on $\mB^{*}$
\begin{equation}\label{lnatural}\langle \varphi x, \psi\rangle=\langle x, \psi\varphi\rangle = \langle x_0,\psi\rangle\langle x,\varphi\rangle
\quad\text{for}\quad x\in \mB^*,\; \varphi,\psi\in \mB\end{equation}
so that  \[\varphi x= \langle x,\varphi\rangle x_0=x\varphi,\]
and then work with   $\mA=_{\mB^*}\!\!\mB$ so that
\begin{equation*}\label{seme}
(x,\varphi)(y,\psi)   = (\varphi y, \varphi\psi) =
(\langle y,\varphi\rangle x_0,\langle x_0,\varphi\rangle \psi).
\end{equation*}

 2) The other algebras $\mA$ which may also be considered based on Example \ref{final} are by taking
\begin{enumerate}
\item $\mB$ with the operation (\ref{z1}) and $\mA=_{\mB^*}\!\!\mB$ with the operation (\ref{seme}),

\item  $\mB$ with the operation (\ref{z2}) and $\mA=\mB_{\mB^\ast}$.
In both of these cases, the reader may verify directly that the algebra $\mA$ is Arens regular.
So these algebras are of no interest to us in this section.

\item $\mB$ with the operation (\ref{z1}) and $\mA=_{\mB^*}\!\!\mB_{\mB^\ast}$,

\item $\mB$ with the operation (\ref{z2}) and $\mA=_{\mB^*}\!\!\mB_{\mB^\ast}$.
In both of these cases, the algebra $\mA$ satisfies Dales conditions but it is not Arens regular.
 Here $\mA$ is not a one-sided ideal in its second dual and has neither a blai nor a blai.
\end{enumerate}
3) We hope to return with more examples in our next paper \cite{semidirect}.
\end{remarks}

\begin{example}
If we adjoin an identity to any of the algebras in the previous example, we obtain a weakly sequentially complete algebra with an identity, which satisfies Dales conditions but it is not Arens regular.  The algebra is not a one-sided ideal in its bidual.
\end{example}


\end{document}